\documentclass[12pt,leqno,oneside]{amsart}
\usepackage{amssymb,amsmath,amscd,graphicx,fontenc,amsthm,mathrsfs, mathtools}
\usepackage[frame, cmtip, arrow, matrix, line, graph, curve]{xy}
\usepackage{graphpap, color}
\usepackage[mathscr]{eucal}
\usepackage{ifthen}
\usepackage{relsize}
\newcommand{\vertiii}[1]{{\left\vert\kern-0.25ex\left\vert\kern-0.25ex\left\vert #1
    \right\vert\kern-0.25ex\right\vert\kern-0.25ex\right\vert}}
\theoremstyle{plain}
\usepackage{geometry}
\usepackage[colorlinks = true,
            linkcolor = blue,
            urlcolor  = blue,
            citecolor = blue,
            anchorcolor = blue]{hyperref}

\usepackage{fancyhdr}
\usepackage{indentfirst}
\usepackage{paralist}
\usepackage{pstricks}
\usepackage{bbm}
\usepackage{tikz}

\newtheorem{theorem}{Theorem}[section]

\newtheorem{lemma}[theorem]{Lemma}
\newtheorem{proposition}[theorem]{Proposition}

\theoremstyle{remark}

\numberwithin{equation}{section}

\newcommand{\NN}{{\mathbb{N}}}
\newcommand{\ZZ}{{\mathbb{Z}}}

\newcommand{\RR}{{\mathbb{R}}}

\DeclareMathOperator{\modu}{mod}

\makeatletter
\def\moverlay{\mathpalette\mov@rlay}
\def\mov@rlay#1#2{\leavevmode\vtop{%
   \baselineskip\z@skip \lineskiplimit-\maxdimen
   \ialign{\hfil$\m@th#1##$\hfil\cr#2\crcr}}}
\newcommand{\charfusion}[3][\mathord]{
    #1{\ifx#1\mathop\vphantom{#2}\fi
        \mathpalette\mov@rlay{#2\cr#3}
      }
    \ifx#1\mathop\expandafter\displaylimits\fi}
\makeatother

\definecolor{cmd}{rgb}{1.0, 0.35, 0.21}

\begin{document}
	\title{Small solutions of quadratic forms with congruence conditions}
	\author{Prasuna Bandi}
	\author{Anish Ghosh}
	\thanks{AG was supported by the Government of India, Department of Science and Technology, Swarnajayanti fellowship DSTSJFMSA-01/2016-17, a CEFIPRA grant, a MATRICS grant, and a grant from the Infosys Foundation. PB and AG acknowledge support of the Department of Atomic Energy, Government of India [under project 12 - R\&D - TFR - 5.01 - 0500].}
	\address{PB and AG: School of Mathematics, Tata Institute of Fundamental Research, Mumbai, 400005, India}
	\email{prasuna@math.tifr.res.in, ghosh@math.tifr.res.in}
	\begin{abstract}
We consider a system of homogeneous quadratic forms with congruence conditions in $n\geq 3$ variables and prove the  existence of two linearly independent integral solutions of bounded height. We also show the existence of small height integral zeros of this system avoiding a given set of hyperplanes.
\end{abstract}
	\maketitle

	\section{Introduction}
	
	Small integral solutions of quadratic forms have been studied extensively beginning with work of Thue \cite{thue}.  A well-known result of J. W. S. Cassels \cite{cas} states that, given an isotropic integral quadratic form $Q$ in $n\ge 2$ variables, there exists $x\in \ZZ^{n}\setminus \{0\}$ such that $$Q(x)=0 \text{ and } \|x\|\ll H^{\frac{n-1}{2}}.$$ 
	Here and henceforth, $H$ denotes the maximum of the absolute values of the coefficients of $Q$ and $\|\cdot\|$ denotes the Euclidean norm on $\RR^{n}$. The implied constant above depends only on $n$. Subsequently, there have been many extensions and generalizations of Cassels's result along different directions, see  for instance \cite{bd, sp, ss, Vaaler}, and \cite{Len 1} for a recent survey. \\
	In \cite{Dav}, Davenport proved that given an isotropic integral quadratic form in $n\ge 2$ variables, there exists two linearly independent solutions $a, b \in \ZZ^{n}$ to $$Q(x)=0$$ such that 
	$$ \|a\| \|b\| \ll H^{n-1}.$$
	In (\cite{kor 1},\cite{kor 2}), Kornhauser proved a variation of Cassels's result for an inhomogeneous quadratic form in $n=2$ and $n\ge 5$  variables. Dietmann \cite{Diet} improved these results and extended them to $3$ and $4$ variables. He used the circle method for $n\ge 5$ variables and geometry of numbers techniques for $3$ and $4$ variables. The proof involves an argument that finding integral solutions of bounded height to inhomogeneous quadratic forms is equivalent to finding integral solutions of bounded height to quadratic form without linear part along with certain congruence conditions.
	Also, there have been several results about the distribution of small height zeros of quadratic forms. In particular Fukshansky \cite{Len} proved the existence of small height zeros of a quadratic form avoiding a given set of hyperplanes. In \cite{Diet 1}, Dietmann improved these bounds.\\ \\
	In this paper, we consider a homogeneous quadratic form with a congruence condition in $n\geq 3$ variables and prove the  existence of two linearly independent integral solutions of bounded height and also small height zeros of this system avoiding a given set of hyperplanes. \\ \\
	

	\section{Main results}
	
	The following Theorem is Proposition 1 of \cite{Diet} when $\kappa=0$.\\
	
	Let $Q$ be a nondegenerate quadratic form in $n\geq 3$ variables with integral coefficients and let $\Delta= \det Q$. Let $\xi \in \ZZ^{n}$ and $\eta \in \NN$. For a prime $p$, define $\lambda(p)$ and $\nu(p)$ to be largest positive integers such that $p^{\lambda(p)}\mid \Delta$ and $p^{\nu(p)}\mid \eta$ respectively. Let $\pi(p)=1$ when $p=2$ and $\pi(p)=0$ otherwise. Define
	$$\Theta:= \prod_{p:\nu(p)>\lambda(p)+\pi(p)}^{} p^{(s-1)\nu(p)-2(s-1)\lambda(p)}$$
	
	\begin{theorem}(\cite{Diet}, Proposition 1) \label{thm d}
		let $\varepsilon>0$. Assume there exists an integral solution to the system
		\begin{equation}\label{eq36}
		\begin{split}
		& Q(x)=0\\
		& x \equiv \xi (\modu \eta)\\
		\end{split}
		\end{equation}
		then there exists $x\in \ZZ^{n}$ satisfying (\ref{eq36}) with
		\begin{equation}\label{eq37}
		\|x\| \ll \left\{ \begin{array}{l}
		\max\{\eta^{3}|\Delta|^{2}H^{2}, \eta^{3}H^{3} \}  \hspace{2.7cm}\;\;\text{ if  }\; n=3\\
		\max \{\eta^{9}|\Delta|^{4}H^{5}, \eta^{20}H^{8} \} \; \hspace{2.5cm}\;\;\text{ if }\; n=4\\
		\eta^{1+\varepsilon +\frac{n}{n-4}} \Theta^{\varepsilon+\frac{2}{n-4}} H^{\varepsilon +\frac{n^{2}-3n+2}{n-4}} |\Delta|^{\varepsilon +\frac{3n+1}{n-4}} \;\;\;\;\text{ if }\; n\geq 5
		\end{array}
		\right.
		\end{equation}
		\end{theorem}
	The following Theorem proves the existence of two small linearly independent integral solutions of system (\ref{eq36}).
	\begin{theorem}\label{thm n}
		Let $Q$ be a nondegenerate quadratic form in $n\geq 3$ variables with integral coefficients. Let $\xi\in \ZZ^{n}$, $\eta \in \NN$ and $\varepsilon>0$. Assume that there is a non-zero integral solution to the system
		\begin{equation}\label{eq20}
		\begin{split}
		& Q(x)=0\\
		& x \equiv \xi (\modu \eta)
		\end{split}
		\end{equation}
		Then there exists two linearly independent integral vectors $a$ and $b$ satisfying (\ref{eq20}) such that 
		\begin{equation}
		\|a\|^{5} \|b\| \ll \left\{ \begin{array}{l}
		\max\{\eta^{38}|\Delta|^{20}H^{3n+18}, \eta^{38}H^{3n+28} \}  \hspace{2.6cm}\text{ if  }\; n=3\\
		\max \{\eta^{98}|\Delta|^{40}H^{3n+48}, \eta^{208}H^{3n+78} \}  \hspace{2.5cm}\text{ if }\; n=4\\
		\eta^{18+\varepsilon +\frac{10n}{n-4}} \Theta^{\varepsilon+\frac{20}{n-4}} H^{\varepsilon +\frac{13n^{2}-44n+28}{n-4}} |\Delta|^{\varepsilon +\frac{10(3n+1)}{n-4}} \hspace{1cm}\text{ if } \; n\geq 5
		\end{array}
		\right.
		\end{equation}
		where the implied constants depends only on $n$.
	\end{theorem}
In the case of $n=3$ we can get much better bounds.
\begin{theorem}\label{thm 3}
	Let $Q$ be a nondegenerate quadratic form on $\RR^{3}$ with integral coefficients. Let $\xi\in \ZZ^{3}$ and $\eta \in \NN$. Assume that there is a non-zero integral solution to the system (\ref{eq20}).
	Then there exists two linearly independent integral vectors $a$ and $b$ satisfying (\ref{eq20}) such that 
	\begin{equation}
	\|a\| \|b\|\ll H^{4}\eta^{6} | \det Q|^{4}
	\end{equation}
	where the implied constants depends only on $n$.
\end{theorem}

The following Theorem proves the existence of small height integral zeros of system (\ref{eq20}) avoiding a given set of hyperplanes.
	\begin{theorem}\label{thm}
		Let $Q$ be a nondegenerate quadratic form and $L_{1},\ldots, L_{k}$ be non zero linear forms in $n\geq 3$ variables with integral coefficients. Let $\xi\in \ZZ^{n}$, $\eta \in \NN$ and $\varepsilon>0$. Assume that there is a non zero integral solution to the system (\ref{eq20}).
		Then there exists $a\in \ZZ^{n}$ satisfying (\ref{eq20}) and such that $L_{i}(a)\neq 0 \text{ for } 1\le i\le k$ and
		\begin{equation}\label{eq43}
		\|a\| \ll \left\{ \begin{array}{l}
		\eta^{3}H^{4}\max\{\eta^{3}|\Delta|^{2}H^{2}, \eta^{3}H^{3} \}  \hspace{2.7cm}\;\;\text{ if  }\; n=3\\
		\eta^{3}H^{4}\max \{\eta^{9}|\Delta|^{4}H^{5}, \eta^{20}H^{8} \} \; \hspace{2.5cm}\;\;\text{ if }\; n=4\\
		\eta^{4+\varepsilon +\frac{n}{n-4}} \Theta^{\varepsilon+\frac{2}{n-4}} H^{4+\varepsilon +\frac{n^{2}-3n+2}{n-4}} |\Delta|^{\varepsilon +\frac{3n+1}{n-4}} \hspace{1.2cm}\text{ if }\; n\geq 5
		\end{array}
		\right.
		\end{equation}
		where the implied constants depends only on $n$ and $k$.
	\end{theorem}
The proof of this Theorem closely follows the proof of case 1 of Theorem from \cite{Diet 1}.


	\section{Preliminary Lemmas}
	\begin{lemma}\label{lem 1}
		Let $\{x_{1},x_{2}, \ldots,x_{n-1}\} $ be a linearly independent subset of $\RR^{n}$. Let $z \in \ZZ^{n}$ and $M \in \NN$. Then there exists $x_{n}\in \ZZ^{n}$ such that $\{x_{1}, x_{2},\ldots, x_{n}\}$ are linearly independent and 
		\begin{equation}\label{eq33}
			x_{n} \equiv z (\modu M) \;\text{ and }\; \|x_{n}\|\ll M.
		\end{equation}
		where the implied constant depends only on $n$.
		
		\begin{proof}
			Let $r\in \ZZ^{n}$ be such that $r \equiv z (\modu M) \text{ and } \|r\|\ll M$. If $\{x_{1}, x_{2},\ldots, r\}$ are linearly independent, then we are done. So assume $\{x_{1}, x_{2},\ldots, x_{n-1}, r\}$ are linearly dependent. Then $r\in \text{span}\{x_{1},\ldots,x_{n-1}\}$. Since $\{x_{1},x_{2}, \ldots,x_{n-1}\} $ are linearly independent, there exists $i\in \{1,\ldots,n\}$ such that $\{x_{1}, x_{2},\ldots, x_{n-1}, e_{i}\}$  are linearly independent where $e_{i}$ is a standard unit vector. Let $x_{n}=r+e_{i}M$. Then $\{x_{1}, x_{2},\ldots, x_{n}\}$ are linearly independent and $x_{n}$ satisfies (\ref{eq33}).
		\end{proof}
	\end{lemma}

The following Lemma is an analogue of Theorem 3.1 of \cite{Len} with congruence conditions.

\begin{lemma}\label{lem 2}
	Let $f(x_{1},\ldots, x_{n}) \in \RR [x_{1},\ldots, x_{n}]$ be a polynomial of degree $m$ which is not identically 0. Let $z \in \ZZ^{n}$ and $M \in \NN$. Then there exists $q\in \ZZ^{n}$ such that $f(q)\neq 0$,
	\begin{equation}\label{eq34}
		q \equiv z (\modu M) \;\text{ and }\; \|q\|\ll M.
	\end{equation}
	 where the implied constant depends only on $m$.
\end{lemma}

\begin{proof}
We proceed by induction on $n$. Suppose $n=1$. Then $f(x_{1})$ is a non-zero polynomial of degree $m$ and hence has atmost $m$ roots. Let $r\in \ZZ$ be such that $r \equiv z (\modu M) \text{ and } |r|<M$. Consider the set $$S=\{r,r+M,r+2M,\ldots, r+mM\}.$$
	Then $|S|=m+1$. Hence there exists an element $q$ of $S$ such that $f(q)\neq 0$. It can be seen easily that $q$ also satisfies (\ref{eq34}). This proves the Lemma when $n=1$.\\
	Now suppose that the Lemma holds for all polynomials in $k$ variables for $1\le k<n$. Let $f$ be a non zero polynomial in $n$ variables. Then there exists $q^{(1)}\in \ZZ^{n-1}$ such that $f(q^{(1)}, x_{n})$ is a non zero polynomial in $x_{n}$. From the case of $n=1$, there exists $q_{n}\in \ZZ$ such that $f(q^{(1)},q_{n})\neq 0$, $q_{n} \equiv z_{n} (\modu M)$  and $|q_{n}|\ll M$.\\
	 Now, define $$g(x_{1},\ldots,x_{n-1})=f(x_{1},\ldots,x_{n-1},q_{n}).$$
	 Then $g$ is a non zero polynomial in $n-1$ variables since $f(q^{(1)},q_{n})\neq 0$. By the induction hypothesis, there exists $q^{(2)}\in \ZZ^{n-1}$ such that $g(q^{(2)})\neq 0$, $q^{(2)}\equiv (z_{1},\ldots,z_{n-1}) (\modu M) $ and $\|q^{(2)}\|\ll M$. Let $q=(q^{(2)},q_{n})$. Then $f(q)\neq 0$ and $q$ satisfies (\ref{eq34}).
\end{proof}


 \section{Proofs of Theorems \ref{thm n} and \ref{thm 3}}
 Theorem \ref{thm n} follows from the following Proposition by observing that $(m,\eta)=1$ implies there exists $m'\in \ZZ$ such that $m m'\equiv 1 (\modu \eta )$ and $|m'|\le \eta$, 
\begin{proposition}\label{prop 1}
	Let $Q$ be a nondegenerate quadratic form in $n\geq 3$ variables with integral coefficients. Let $\xi\in \ZZ^{n}$ and $\eta \in \NN$. Let $\varepsilon>0$. Assume that there is a non-zero integral solution to the system
	\begin{equation}\label{eq38}
	\begin{split}
	& Q(x)=0\\
	& x \equiv m\xi (\modu \eta)
	\end{split}
	\end{equation}
	for some $m\in \ZZ$ with $(m,\eta)=1$.
	Then there exists two linearly independent integral vectors $a$ and $b$ satisfying (\ref{eq38}) such that 
	\begin{equation}\label{eq44}
	\|a\|^{5} \|b\| \ll \left\{ \begin{array}{l}
	\max\{\eta^{32}|\Delta|^{20}H^{3n+18}, \eta^{32}H^{3n+28} \}  \hspace{2.6cm}\text{ if  }\; n=3\\
	\max \{\eta^{92}|\Delta|^{40}H^{3n+48}, \eta^{202}H^{3n+78} \}  \hspace{2.5cm}\text{ if }\; n=4\\
	\eta^{12+\varepsilon +\frac{10n}{n-4}} \Theta^{\varepsilon+\frac{20}{n-4}} H^{\varepsilon +\frac{13n^{2}-44n+28}{n-4}} |\Delta|^{\varepsilon +\frac{10(3n+1)}{n-4}} \hspace{1cm}\text{ if } \; n\geq 5
	\end{array}
	\right.
	\end{equation}
	where the implied constants depend only on $n$.
\end{proposition}
\begin{proof}
	Suppose $0$ is a solution of (\ref{eq38}). This implies $\eta \mid \xi$. Since $Q(x)=0$ has a non-zero integral solution, by Theorem 1 of \cite{Dav} there exist two linearly independent integral solutions $u, v$ of $Q(x)=0$ such that 
	\begin{equation*}
	\|u\| \|v\|\ll H^{n-1}
	\end{equation*}
	Let $a=\eta u, b=\eta v$. Then $a, b$ are two linearly independent integral vectors satisfying (\ref{eq38}) and 
	\begin{equation*}
	\|a\| \|b\| \ll \eta^{2} H^{n-1}
	\end{equation*}
	Thus (\ref{eq44}) holds.\\
	Now assume that $0$ is not a solution of (\ref{eq38}). Then $\xi\neq 0$. Without loss of generality, we may assume that $\xi$ is primitive. Let $a$ be an integral vector satisfying (\ref{eq38}) such that $\|a\|$ is least. Then by Theorem \ref{thm d}, we have
	\begin{equation}\label{eq21}
	\|a\| \ll \left\{ \begin{array}{l}
	\max\{\eta^{3}|\Delta|^{2}H^{2}, \eta^{3}H^{3} \}  \hspace{2.7cm}\text{ if  }\; n=3\\
	\max \{\eta^{9}|\Delta|^{4}H^{5}, \eta^{20}H^{8} \} \; \hspace{2.5cm}\text{ if }\; n=4\\
	\eta^{1+\varepsilon +\frac{n}{n-4}} \Theta^{\varepsilon+\frac{2}{n-4}} H^{\varepsilon +\frac{n^{2}-3n+2}{n-4}} |\Delta|^{\varepsilon +\frac{3n+1}{n-4}} \;\;\;\;\text{ if }\; n\geq 5
	\end{array}
	\right.
	\end{equation}
	Since $a\neq 0$, we get that $a$ is primitive. Indeed, for a prime $p$ if $p\mid a$, then $(p,\eta)=1$ since $\xi$ is primitive. This implies that $(1/p)a$ satisfies (\ref{eq38}) contradicting the minimality of $\|a\|$. \\
	Let $E$ be the $n-1$ dimensional space through origin perpendicular to the vector $a$. Denote by $\Gamma$ the projection of $\ZZ^{n}$ onto $E$. Then $\Gamma$ is a lattice in $E$ of determinant $\|a\|^{-1}$. Denote by $q_{1},\ldots, q_{n-1}$ the successive minima of the $n-1$ dimensional lattice $\Gamma$. Then
	\begin{equation*}
	\|q_{1}\|=\min \{\|q\|: q\in \Gamma\setminus\{0\}\}
	\end{equation*}
	\begin{equation*}
	\|q_{i}\|=\min \{\|q\|: q\in \Gamma \text{ and } q,q_{1},\ldots,q_{i-1} \text{ are linearly independent}\} \text{ for }  2\leq i\leq n-1
	\end{equation*}
	\begin{equation} \label{eq39}
		\|q_{1}\|\le \|q_{2}\|\le \ldots \le \|q_{n-1}\|.
	\end{equation}
	By Minkowski's second theorem, we have
	\begin{equation}\label{eq22}
	\|q_{1}\| \cdots \|q_{n-1}\| \ll \|a\|^{-1}
	\end{equation}
	By definition of $\Gamma$, for $1\le i\le n-1$ there exists $t_{i} \in \ZZ^{n}$ such that
	\begin{equation}\label{eq24}
		t_{i}=q_{i}+\alpha _{i} a \;\text{ for some }\alpha_{i}\in \RR \text{ with }|\alpha_{i}|\leq 1.
	\end{equation}
	Since $a, q_{i}$ are linearly independent, we have
	\begin{equation}\label{eq50}
		\|t_{i}\|^{2} \le \|q_{i}\|^{2}+\|a\|^{2}.
	\end{equation}
	Since $\{a,q_{1},\ldots, q_{n-1}\}$ are linearly independent, we get that $\{a,t_{1},\ldots, t_{n-1}\}$ are linearly independent.
	Now, consider the linear subspace 
	$$S:=\{x:Q(a,x)=0\}.$$ 
	Since $Q$ is nondegenerate, $t_{i}\notin S$ for some $i$. Denote by $r$ the least positive integer such that $t_{r}\notin S$. Then $Q(a, t_{i})=0$ for all $1\le i\le r-1$ and $Q(a, t_{r})\neq 0$.\\ \\
	\textbf{Case 1}: Suppose there exists $i\in \{1,\ldots,r-1\}$ such that $Q(t_{i})=0$.\\
	Let $j$ be the least such that $Q(t_{j})=0$. Then $Q(t_{k})\neq 0$ for $1\le k\le j-1$. Since $Q(a,t_{k})=0$, we have $Q(q_{k})=Q(t_{k})$. Therefore
	\begin{equation*}
		H \|q_{k}\|^{2}\gg |Q(q_{k})|=|Q(t_{k})| \geq 1
	\end{equation*}
	which implies 
	\begin{equation}\label{eq23}
		\|q_{k}\|\gg H^{-1/2} \text{ for } 1\leq k\leq j-1. 
	\end{equation}
	Using (\ref{eq39}), (\ref{eq22}) and (\ref{eq23}), we get 
	\begin{equation*}
		(H^{-1/2})^{j-1} \|q_{j}\|^{n-j} \ll \|a\|^{-1}
	\end{equation*}
	which implies
	\begin{equation}\label{eq25}
		\|q_{j}\|\ll H^{\frac{j-1}{2(n-j)}} \|a\|^{\frac{-1}{n-j}}.
	\end{equation}
	Now, let $b=a+\eta t_{j}$. Then $b$ satisfies (\ref{eq38}) and $a, b$ are linearly independent.
	\begin{align*}
	 \|a\| \|b\|=\|a\| \|a+\eta t_{j}\| &\le \|a\|^{2}+\eta \|a\| \|t_{j}\|\\
	 & \le \|a\|^{2}+\eta \|a\| (\|q_{j}\|+\|a\|) \; (\text{ by }\ref{eq24})\\
	 & \ll \eta \|a\|^{2}+\eta H^{\frac{j-1}{2(n-j)}} \|a\|^{1-\frac{1}{n-j}} \; (\text{ by } \ref{eq25})\\
	 &\ll \eta \|a\|^{2}+\eta H^{\frac{n-3}{4}} \|a\|^{1-\frac{1}{n-1}}
	\end{align*}
	which gives
		\begin{equation*}
	\|a\| \|b\| \ll \left\{ \begin{array}{l}
	\eta H^{(n-3)/4}\|a\|^{2} \hspace{0.6cm} \text{ if } \|a\| > 1,\\
	\eta H^{(n-3)/4} \hspace{1.4cm} \text{ if } \|a\| \le 1.
	\end{array}
	\right.
	\end{equation*}
	Using (\ref{eq21}) to bound $\|a\|$, it is easy to see that (\ref{eq44}) holds.\\
	\textbf{Case 2}: Suppose that $Q(t_{i})\neq 0$ for all $1\le i\le r-1$.\\
	Then
	\begin{equation*}
	H \|q_{i}\|^{2}\gg |Q(q_{i})|=|Q(t_{i})| \geq 1
	\end{equation*}
	which implies 
	\begin{equation}\label{eq26}
	\|q_{i}\|\gg H^{-1/2} \text{ for } 1\leq i\leq r-1. 
	\end{equation}
	Using (\ref{eq39}), (\ref{eq22}) and (\ref{eq26}), we get 
	\begin{equation*}
	(H^{-1/2})^{r-1} \|q_{r}\|^{n-r} \ll \|a\|^{-1}
	\end{equation*}
	which implies
	\begin{equation}\label{eq27}
	\|q_{r}\|\ll H^{\frac{r-1}{2(n-r)}} \|a\|^{\frac{-1}{n-r}}.
	\end{equation}
	Hence,
	\begin{equation}
		\|q_{r}\| \ll \left\{ \begin{array}{l}
		H^{(n-2)/2}\|a\|^{-1/(n-1)} \hspace{0.5cm} \text{ if } \|a\|> 1,\\
		H^{(n-2)/2}\|a\|^{-1} \hspace{1.5cm} \text{ if } \|a\| \le 1.
		\end{array}
		\right.
	\end{equation}
	 \textbf{Case 2.1} Assume that $Q(t_{r})\neq 0$.\\
	 Let
	$$x^{(1)}=Q(t_{r})a-2Q(t_{r},a)t_{r}.$$
Then $Q(x^{(1)})=0$ and $a,x^{(1)}$ are linearly independent. \\
By (\ref{eq24}), we have $$x^{(1)}=Q(q_{r})a-2Q(q_{r},a)q_{r}$$
and hence, 
\begin{equation} \label{eq30}
	\|x^{(1)}\|\ll H \|q_{r}\|^{2} \|a\|.
\end{equation}
For a prime $p\mid \eta$, define $l(p)$ to be the largest integer such that $p^{l(p)} \mid Q(t_{r})$. Then 

\begin{equation} \label{eq28}
	\prod_{p\mid \eta} p^{l(p)} \le |Q(t_{r})|.
\end{equation}
Let 
\begin{equation*}
	M= \eta \prod_{p\mid \eta} p^{2 l(p)} \text{ and } N=\prod_{p\mid \eta} p^{2 l(p)}.
\end{equation*}
By Lemma \ref{lem 1}, there exists $z'\in \ZZ^{n}$ such that $z'\equiv t_{r}(\textrm{mod}\ M)$, $\{a, x^{(1)}, z'\}$ are linearly independent and 
\begin{equation}\label{eq29}
\|z'\|\ll M. 
\end{equation}
Now, define
$$x^{(2)}=Q(z')x^{(1)}-2 Q(z',x^{(1)})z'.$$
Then $Q(x^{(2)})=0$ and 
$$x^{(2)}\equiv Q(t_{r})^{2}a \  (\textrm{mod}\ M).$$
This implies that
$$ N^{-1}x^{(2)}\equiv ma \ (\textrm{mod}\ \eta)$$
for some $m\in \ZZ$ with $(m,\eta)=1$. Since $a$ satisfies the congruence condition in (\ref{eq38}), we get
$$ N^{-1}x^{(2)}\equiv m\xi \ (\textrm{mod}\ \eta)$$
for some $m\in \ZZ$ with $(m,\eta)=1$.
Let $b=N^{-1}x^{(2)}$. Then $b$ is an integral solution of the system (\ref{eq38}). This implies $b\neq 0$; hence both $Q(z'), Q(z',x^{(1)})$ cannot be zero. This shows that $a, b$ are linearly independent since $\{a, x^{(1)}, z'\}$ are linearly independent. Now,
\begin{align*}
\|a\| \|b\| &= N^{-1} \|a\| \|x^{(2)} \|\\
&\ll H N^{-1} \|a\|\|x^{(1)}\| \|z'\|^{2}\\
&\ll H^{2} N^{-1}  M^{2}\|q_{r}\|^{2}\|a\|^{2} \; (\text{by } (\ref{eq30}), (\ref{eq29}) )\\
&\ll H^{2} \eta^{2} |Q(t_{r})|^{2} \|q_{r}\|^{2}\|a\|^{2} \; (\text{by } (\ref{eq28}))\\
&\ll H^{4} \eta^{2} \|t_{r}\|^{4} \|q_{r}\|^{2}\|a\|^{2}. 
\end{align*}
Using (\ref{eq50}), we get
\begin{equation*}
\|a\| \|b\| \ll \left\{ \begin{array}{l}
\eta^{2} H^{3n-2}\|a\|^{\frac{6n-8}{n-1}} \hspace{0.6cm} \text{ if } \|a\| > 1\\
\eta^{2} H^{3n-2} \|a\|^{-4} \hspace{0.9cm} \text{ if } \|a\| \le 1
\end{array}
\right.
\end{equation*}
which gives 
\begin{equation*}
\|a\|^{5} \|b\| \ll \left\{ \begin{array}{l}
\eta^{2} H^{3n-2}\|a\|^{10} \hspace{0.5cm} \text{ if } \|a\| > 1\\
\eta^{2} H^{3n-2}  \hspace{1.5cm} \text{ if } \|a\| \le 1.
\end{array}
\right.
\end{equation*}
Using (\ref{eq21}) to bound $\|a\|$, it is easy to see that (\ref{eq44}) holds.\\
\textbf{Case 2.2}: Assume that $Q(t_{r})=0$.\\
Let $z=a+t_{r}$. Then $Q(z)=2 Q(a,t_{r})\neq 0$. For a prime $p\mid \eta$, define $l(p)$ to be the largest integer such that $p^{l(p)} \mid Q(z)$. Then 

\begin{equation} \label{eq32}
\prod_{p\mid \eta} p^{l(p)} \le |Q(z)|.
\end{equation}
Let 
\begin{equation*}
	M= \eta \prod_{p\mid \eta} p^{l(p)} \text{ and } N=\prod_{p\mid \eta} p^{l(p)}.
\end{equation*}
By Lemma \ref{lem 1}, there exists $z'\in \ZZ^{n}$ such that $z'\equiv z(\modu M)$, $\{a, t_{r}, z'\}$ are linearly independent and 
\begin{equation}\label{eq31}
\|z'\|\ll M. 
\end{equation}
Now, define
$$x^{(2)}=Q(z') t_{r}-2 Q(z',t_{r})z'.$$
It is easy to check that $Q(x^{(2)})=0$ and 
$$x^{(2)}\equiv -Q(z) a \  (\modu M).$$
This implies
$$ N^{-1}x^{(2)}\equiv ma \ (\modu \eta)$$
for some $m\in \ZZ$ with $(m,\eta)=1$. Since $a$ satisfies the congruence condition in (\ref{eq38}), we get
$$ N^{-1}x^{(2)}\equiv m\xi \ (\modu \eta)$$
for some $m\in \ZZ$ with $(m,\eta)=1$.
Let $b=N^{-1}x^{(2)}$. As in Case 2.1, it follows that $b$ is an integral solution of the system (\ref{eq38}) and $a, b$ are linearly independent. Now,
\begin{align*}
\|a\| \|b\|&= N^{-1} \|a\| \|x^{(2)} \|\\
&\ll H N^{-1} \|a\|\|t_{r}\| \|z'\|^{2}\\
&\ll H N^{-1}  M^{2}\|a\|\|t_{r}\|\; (\text{by } (\ref{eq31}))\\
&\ll H \eta^{2} |Q(z)| \|a\|\|t_{r}\| \; (\text{by } (\ref{eq32}))\\
&\ll H^{2} \eta^{2} \|z\|^{2} \|a\|\|t_{r}\|.
\end{align*}
Hence, by (\ref{eq50}) we get
\begin{equation*}
\|a\| \|b\| \ll \left\{ \begin{array}{l}
\eta^{2} H^{\frac{3n-2}{2}}\|a\|^{4} \hspace{0.8cm} \text{ if } \|a\| > 1\\
\eta^{2} H^{\frac{3n-2}{2}} \|a\|^{-2} \hspace{0.6cm} \text{ if } \|a\| \le 1
\end{array}
\right.
\end{equation*}
which further gives
\begin{equation*}
\|a\|^{5} \|b\| \ll \left\{ \begin{array}{l}
\eta^{2} H^{\frac{3n-2}{2}}\|a\|^{8} \hspace{0.7cm} \text{ if } \|a\| > 1\\
\eta^{2} H^{\frac{3n-2}{2}}  \hspace{1.5cm} \text{ if } \|a\| \le 1
\end{array}
\right.
\end{equation*}
and now using (\ref{eq21}) it can be checked that (\ref{eq44}) holds.\\
\end{proof}

Now, we prove Theorem \ref{thm 3}. It is easy to see that it follows from the Proposition below. The main difference in the proof of Proposition \ref{prop 3} from that of Proposition \ref{prop 1} is we get better bounds in (\ref{eq 3}) in place of (\ref{eq28}), following the proof of Proposition 3 from \cite{Diet}.
\begin{proposition}\label{prop 3}
	Let $Q$ be a nondegenerate quadratic form on $\RR^{3}$ with integral coefficients. Let $\xi\in \ZZ^{3}$ and $\eta \in \NN$. Assume that there is a non-zero integral solution to the system
	\begin{equation}\label{eq}
	\begin{split}
	& Q(x)=0\\
	& x \equiv m\xi (\modu \eta)
	\end{split}
	\end{equation}
	for some $m\in \ZZ$ with $(m,\eta)=1$.
	Then there exists two linearly independent integral vectors $a$ and $b$ satisfying (\ref{eq}) such that 
	\begin{equation}
	\|a\| \|b\|\ll H^{4}\eta^{4} | \det Q|^{4}
	\end{equation}
\end{proposition}
\begin{proof}
	Suppose $0$ is a solution of (\ref{eq}). Then by following the same proof as in Proposition \ref{prop 1} we get that there exists two linearly independent integral vectors $a$, $b$ satisfying (\ref{eq}) and 
	\begin{equation}\label{eq47}
	\|a\| \|b\|\ll \eta^{2} H^{2}.
	\end{equation}
	Now assume that $0$ is not a solution of (\ref{eq}). Without loss of generality assume that $\xi$ is primitive. Let $a$ be an integral vector satisfying (\ref{eq}) such that $\|a\|$ is least. Then 
	\begin{equation}\label{eq 2}
	\|a\|\ll \eta^{2} | \det Q |^{2} H^{2} \;(\text{ by Proposition 3 of \cite{Diet}}).
	\end{equation}
	Let $E$ be the plane through the origin perpendicular to $a$. Denote by $\Gamma$ the projection of $\ZZ^{3}$ onto $E$. Since $a$ is primitive, $\Gamma$ is a lattice in $E$ of determinant $\|a\|^{-1}$. Denote by $q_{1}, q_{2}$ the successive minima of the lattice $\Gamma$. 
	 By Minkowski's second theorem, we have
	\begin{equation}\label{min}
	\|q_{1}\|\|q_{2}\|\ll \|a\|^{-1}.
	\end{equation}
	By definition of $\Gamma$, there exists $t_{1},t_{2} \in \ZZ^{3}$ whose projections onto $E$ are $q_{1},q_{2}$ respectively. Then $\{a,t_{1},t_{2}\}$ are linearly independent.
	Denote by $S:=\{x:Q(a,x)=0\}$. Since $Q$ is nondegenerate both $t_{1},t_{2}$ cannot lie in $S$.\\
	\textbf{Case 1} ($t_{1}\notin S$): Then $Q(a,t_{1})\neq 0$.\\
	For a prime $p$, let $\lambda(p)$ be largest positive integer such that $p^{\lambda(p)}\mid \det Q$. Now choose $\alpha \in \ZZ$ such that $p^{\lambda(p)+\pi(p)+1}\nmid Q(z)$ for all primes $p\mid \eta$ where $z=\alpha a+t_{1}$, $\pi(p)=1 \text{ when } p=2 \text{ and }\pi(p)=0$ otherwise (see page 572 of \cite{Diet}).
	Since $Q(a,t_{1})\neq 0$, we have $Q(a,z)\neq 0$. Let
	$$x^{(1)}=Q(z)a-2Q(z,a)z.$$
	Then $Q(x^{(1)})=0$ and $a,x^{(1)}$ are linearly independent. But $x^{(1)}$ may not satisfy the congruence condition in (\ref{eq}). Denote by $L$ the line joining $0$ and $a$. Then
	\begin{align*}
	\text{dist}(z,L)&=\text{dist}(\alpha a+t_{1}, L)\\
	&=\text{dist}(t_{1},L)=\|q_{1}\|\\
	&\le \|a\|^{-1/2}\ \ ( \text{by } \ref{min})
	\end{align*}
	which implies
	\begin{equation*}
	\sum_{i,j=1}^{3}(a_{i}z_{j}-a_{j}z_{i})^{2} \le \|a\|.
	\end{equation*}
	Now, by Lemma 20 of \cite{Diet} we get
	\begin{equation}\label{eq 1}
	\|x^{(1)}\|\ll H.
	\end{equation}
	For a prime $p\mid \eta$, define $l(p)$ to be the largest integer such that $p^{l(p)} \mid Q(z)$.  Since $p^{\lambda(p)+\pi(p)+1}\nmid Q(z)$ for all $p\mid \eta$, we get
	$l(p)\le \lambda(p)+\pi(p)$. Hence
	\begin{equation}\label{eq 3}
	\prod_{p\mid \eta} p^{ l(p)}\le 2 | \det Q |.
	\end{equation}
	Let 
	$$M= \eta \prod_{p\mid \eta} p^{2 l(p)} \text{ and }N=\prod_{p\mid \eta} p^{2 l(p)}.$$
	By Lemma \ref{lem 1}, there exists $z'\in \ZZ^{3}$ such that $z'\equiv z(\modu M)$, $\{a, x^{(1)}, z'\}$ are linearly independent and $\|z'\|\ll M $.
	Now, define
	$$x^{(2)}=Q(z')x^{(1)}-2 Q(z',x^{(1)})z'.$$
	Then $Q(x^{(2)})=0$ and 
	$$x^{(2)}\equiv Q(z)^{2}a \  (\modu M)$$
	which implies
	$$ N^{-1}x^{(2)}\equiv m\xi \ (\modu \eta)$$
	for some $m\in \ZZ$ with $(m,\eta)=1$. 
	Let $b=N^{-1}x^{(2)}$. Then $b$ satisfies (\ref{eq}) and $a, b$ are linearly independent.
	Now, using (\ref{eq 2}), (\ref{eq 1}) and (\ref{eq 3}) it can be checked that
	\begin{equation}\label{eq46}
	 \|a\| \|b\| \ll H^{4} \eta^{4} | \det Q|^{4}. 
	\end{equation}
	\textbf{Case 2}($t_{1}\in S$): Then $t_{2}\notin S$. Hence $Q(a,t_{1})=0$ and $Q(a,t_{2})\neq 0$. \\
	As in Case 1, choose $\alpha \in \ZZ$ such that $p^{\lambda(p)+\pi(p)+1}\nmid Q(z)$ for all primes $p\mid \eta$ where $z=\alpha a+t_{2}$, $\pi(p)=1 \text{ when } p=2 \text{ and }\pi(p)=0$ otherwise. 
	Let
	$$x^{(1)}=Q(z)a-2Q(z,a)z.$$
	Then $Q(x^{(1)})=0$ and $a,x^{(1)}$ are linearly independent since $Q(a,z)\neq 0$. Now,
	\begin{align*}
	\text{dist}(z,L)&=\text{dist}(\alpha a+t_{2}, L)\\
	&=\text{dist}(t_{2},L)=\|q_{2}\|      
	\end{align*}
	which gives
	\begin{equation*}
	\sum_{i,j=1}^{3}(a_{i}z_{j}-a_{j}z_{i})^{2} \le \|a\|^{2}\|q_{2}\|  ^{2}
	\end{equation*}
	and by Lemma 20 of \cite{Diet} we have that
	\begin{equation}\label{eq 9}
	\begin{split}
	\|x^{(1)}\|& \ll H \|q_{2}\| ^{2} \|a\|\\
	& \ll H \|q_{1}\|   ^{-2} \|a\|^{-1} \;(\text{ by } \ref{min}).
	\end{split}
	\end{equation}
	In order to bound $x^{(1)}$, we need to find a lower bound for $q_{1}$.\\\\
	Claim: $Q(q_{1})\neq 0$.\\
	If not, since $Q(a,q_{1})=0$ and $Q(a)=0$, we get that $Q$ vanishes on the space spanned by $a, q_{1}$ which is $2$ dimensional. But the dimension of any totally isotropic subspace in a $3$ dimensional quadratic space has to be less than or equal to $1$ which is a contradiction. Hence $Q(q_{1})\neq 0$\\ \\
	Since $Q(a,q_{1})=0$, we get that $Q(t_{1})=Q(q_{1})\neq 0$. Since $Q(t_{1})\in \ZZ$, we have $|Q(q_{1})|\geq 1$. This implies $H \|q_{1}\|^{2}\gg 1$ which further gives $\|q_{1}\|\gg H^{-1/2}$. Hence 
	\begin{equation*}
	\|x^{(1)}\| \ll H^{2} \|y\|^{-1} \;(\text{ by } \ref{eq 9}).
	\end{equation*}
	By proceeding as in Case 1, we find an integral vector $b$ satisfying (\ref{eq}), such that $\{a, b\}$ are linearly independent and
	\begin{equation}\label{eq45}
	\|a\|\|b\|\ll H^{3} \eta^{2} | \det Q|^{2}.
	\end{equation}
	Now by (\ref{eq47}), (\ref{eq46}) and (\ref{eq45}) proposition follows.

\end{proof} 

\section{Proof of Theorem \ref{thm}}

\begin{proof}
	By Theorem \ref{thm d}, there exists $x\in \ZZ^{n}$ satisfying (\ref{eq36}) and (\ref{eq37}).\\
	\textbf{Case 1}: Suppose $x=0$.\\
	This implies $\eta \mid \xi$. By Case 1 of Theorem of \cite{Diet 1}, there exists $b\in \ZZ^{n}$ such that $Q(b)=0$, $L_{i}(b)\neq 0$ and 
	$$\|b\| \ll H^{\frac{n+1}{2}}.$$
	Let $a=\eta b$. Then $a$ satisfies (\ref{eq20}), $L_{i}(a)\neq 0 \text{ for } 1\le i\le k$ and 
	\begin{equation}\label{eq42}
	\|a\| \ll \eta H^{\frac{n+1}{2}}.
	\end{equation}
	From (\ref{eq42}) it can be easily seen that $a$ satisfies (\ref{eq43}).\\
	\textbf{Case 2}: Suppose $x\neq 0$.\\
	By Lemma \ref{lem 2}, there exists $t\in \ZZ^{n}$ such that $Q(t)\neq 0$ and $\|t\|\ll 1$. Let
	\begin{equation*}
	z=Q(t)x-2Q(x,t)t.
	\end{equation*}
	Then $Q(z)=0$ and by Lemma 19 of \cite{Diet} we get that $z\neq 0$.\\
	For a prime $p\mid \eta$, define $l(p)$ to be the largest positive integer such that $p^{l(p)} \mid Q(t)$. Then 
	\begin{equation} \label{eq40}
	\prod_{p\mid \eta} p^{l(p)} \le |Q(t)|.
	\end{equation}
	Let 
	\begin{equation}\label{eq41}
	M= \eta \prod_{p\mid \eta} p^{2 l(p)} \;\text{ and }\; N=\prod_{p\mid \eta} p^{2 l(p)}. 
	\end{equation}
	For $y\in \RR^{n}$, define 
	$$u_{y}:=Q(y)z-2Q(y,z)y.$$
	Claim: There exists $y\in \RR^{n}$ such that $L_{i}(u_{y})\neq 0$.\\
	If not, then $L_{i}(u_{y})=0$ for all $y\in \RR^{n}$. If $L_{i}(z)\neq 0$, we get
	$$ Q(y)=\frac{2Q(y,z)L_{i}(y)}{L_{i}(z)} \text{ for all } y\in \RR^{n}.$$
	Thus $Q$ is a product of two linear forms, which is a contradiction. \\
	If $L_{i}(z)=0$, then $Q(y,z)L_{i}(y)=0$ for all $y\in \RR^{n}$ which is not possible since $z\neq 0$ and $L_{i}$ is a non zero linear form.\\
	Now, define $f(y):=L_{1}(u_{y})\ldots L_{k}(u_{y})$. Then by the above Claim, it follows that $f(y)$ is a non zero polynomial. By Lemma \ref{lem 2}, there exists $q\in \ZZ^{n}$ such that $f(q)\neq 0$, $q\equiv t (\modu M)$ and $\|q\| \ll M$. Hence
	\begin{equation*}
	\begin{split}
	u_{q}&=Q(q)z-2Q(q,z)q\\
	&\equiv Q(t)^{2}x \;(\modu M).
	\end{split}
	\end{equation*}
	It follows that $N^{-1}u_{q}\equiv m\xi (\modu \eta)$ for some $m\in \ZZ$ with $(m,\eta)=1$. Let $m'\in \ZZ$ be such that $m m' \equiv 1 (\modu \eta)$ and $|m'|\leq \eta$. Let $a=m' N^{-1}u_{q}$. Then $a$ satisfies (\ref{eq20}) and since $f(q)\neq 0$ we have $L_{i}(a)\neq 0$ for $1\le i\le k$. Finally,
	\begin{equation*}
	\begin{split}
	\|a\|&=\|m'N^{-1} u_{q}\| \le \eta N^{-1}\|u_{q}\|\\
	&\ll \eta N^{-1} H \|q\|^{2}\|z\|\\
	& \ll \eta M^{2}N^{-1}H^{2}\|t\|^{2}\|x\| \\
	& \ll \eta^{3} H^{4} \|x\| \; \;(\text{by } (\ref{eq40}), (\ref{eq41})).
	\end{split}
	\end{equation*}
	Now, by using (\ref{eq37}) to bound $\|x\|$ it can be seen that $a$ satisfies (\ref{eq43}).
	
\end{proof}



\begin{thebibliography}{1}
	\bibitem{bd} B. J. Birch and H. Davenport, \textit{Quadratic equations in several variables}, Proc. Cambridge Philos. Soc. 54 (1958), 135--138.
	\bibitem{cas} J. W. S. Cassels, \textit{Bounds for the least solutions of homogeneous quadratic equations}, Proc. Cambridge Philos. Soc. 51, 262--264.
	\bibitem{Len}  L. Fukshansky, \textit{Small zeros of quadratic forms with linear conditions}, J. Number Theory 108 (2004), 29--43.
	\bibitem{Len 1} L. Fukshansky, \textit{Heights and quadratic forms: on Cassels’ theorem and its generalizations}, In
        	W. K. Chan, L. Fukshansky, R. Schulze-Pillot, and J. D. Vaaler, editors, Diophantine methods,
        	lattices, and arithmetic theory of quadratic forms, Contemp. Math., 587, 77--94. Amer.
        	Math. Soc., Providence, RI, 2013.
        	\bibitem{Dav} H. Davenport, \textit{Homogeneous quadratic equations}, Mathematika 18, 1--4.
        	\bibitem{Diet} R. Dietmann, \textit{Small solutions of quadratic diophantine equations}, Proc. London Math. Soc. (3) 86, 545--582.
	\bibitem{Diet 1} R. Dietmann, \textit{Small zeros of quadratic forms avoiding a finite number of prescribed hyperplanes}, Canad. Math. Bull., 52(1), 63--65, 2009.
         \bibitem{kor 1} D. M. Kornhauser, \textit{On small solutions of the general nonsingular quadratic Diophantine equation in five and more unknowns}, Math. Proc. Cambridge Philos. Soc. 107, 197--211.
        	\bibitem{kor 2} D. M. Kornhauser, \textit{On the smallest solution to the general binary quadratic equation}, Acta Arith. 55, 83--94.
        \bibitem{ss} H. P. Schlickewei and W. M. Schmidt, \textit{Quadratic geometry of numbers}, Trans. Amer. Math. Soc. 301 (1987), no. 2, 679--690.
        \bibitem{sp} R. Schulze-Pillot, \textit{Small linearly independent zeros of quadratic forms}, Monatsh. Math. 95 (1983), 241--249.
        \bibitem{thue} A. Thue, \textit{Eine Eigenschaft der Zahlen der Fermat'schen Gleichung}, Skr. VidenskSdsk.,
Christ., (Mat.-naturv. Kl.), 1911.
        \bibitem{Vaaler} J. D. Vaaler, \textit{Small zeros of quadratic forms over number fields}, Trans. Amer. Math. Soc. 302(1987), no.
1, 281--296.
	
\end{thebibliography}
	\end{document}